\documentclass[12pt]{article}

\usepackage{amsmath}
\usepackage{latexsym}
\usepackage{amsfonts}
\usepackage{amssymb}
\usepackage{amscd}
\usepackage{theorem}
\newtheorem{theorem}{Theorem}[section]

\newtheorem{lemma}[theorem]{Lemma}
\newtheorem{proposition}[theorem]{Proposition}
\newtheorem{corollary}[theorem]{Corollary}

{\theorembodyfont{\rmfamily}

  \newtheorem{remark}[theorem]{Remark}
}

\newenvironment{proof}{    
  \noindent
  \textbf{Proof.}}{
  \hfill $\Box$
  \vspace{3mm}
}

\numberwithin{equation}{section}

\newcommand{\N}{\mathbb{N}} 
\newcommand{\R}{\mathbb{R}} 
\newcommand{\C}{\mathbb{C}} 
\newcommand{\D}{\mathbb{D}} 

\begin{document}

\title{A note about Volterra operators on weighted Banach spaces of entire functions}

\author{Jos\'{e} Bonet  and Jari Taskinen}

\date{}

\maketitle

\begin{abstract}
We characterize boundedness, compactness and weak compactness of Vol\-terra operators $V_g$ acting between different weighted Banach spaces $H_v^{\infty}(\C)$ of entire functions with sup-norms in terms of the symbol $g$; thus we complement recent work by Bassallote, Contreras, Hern\'andez-Mancera, Mart\'{\i}n and Paul \cite{BCHMP} for spaces of holomorphic functions on the disc and by Constantin and Pel\'aez \cite{CP} for reflexive weighted Fock spaces.
\end{abstract}

\renewcommand{\thefootnote}{}
\footnotetext{\emph{Acknowledgement.} The research of J.\,Bonet was partially supported by MEC and FEDER Project MTM2010-15200 and by GV Project Prometeo II/2013/013. The research of J.\,Taskinen was 
partially supported by the V{\"a}is{\"a}l{\"a} Foundation of the Finnish Academy of Sciences and
Letters.}


\section{Introduction, notation and preliminaries}
\label{sec1}

The aim of this paper is to investigate boundedness and (weak) compactness of the Volterra operator when it acts between two weighted Banach spaces of entire functions $H^\infty_v(\C)$ and $H^\infty_w(\C)$.  We reduce the problem to the study of multiplication operators between related weighted Banach spaces of entire functions, so our approach is similar to that in \cite{BCHMP}. It enables us  to give 
simplified proofs of some results of \cite{CP} for the Volterra operator between two weighted Fock spaces of order infinity  and to obtain new results, in particular about weak compactness and about operators on the smaller spaces $H_v^0(\C)$ and new examples. Our main results are Theorems 
\ref{Volterracont}--\ref{th3.5a}, and concrete examples and applications  are given
Remark \ref{rem3.12} and its corollaries. In Lemma \ref{lem3.5} we also study some general 
properties of weight functions relevant for the  theory of $H^\infty_v(\C)$-spaces: we
formulate sufficient conditions for the so called essentialness of weight functions. 

In what follows $H(\C)$ and $\mathcal{P}$ denote the space of entire functions and the space of polynomials, respectively. The space $H(\C)$ will be endowed with the compact open topology $\tau_{co}.$ The differentiation operator $Df(z)=f'(z)$ and
the integration operator $Jf(z)=\int_0^z f(\zeta)d\zeta$ are continuous on $H(\C).$

Given an entire function $g \in H(\C)$, the Volterra operator $V_g$ with symbol $g$ is defined on $H(\C)$ by $$
V_g(f)(z):= \int_0^z f(\zeta)g'(\zeta)d\zeta \ \ \ \ (z \in \C).
$$
For $g(z)=z$ this reduces to the integration operator, denoted by $J$. Clearly $V_g$ defines a continuous operator on $H(\C)$.
The Volterra operator for holomorphic functions on the unit disc was introduced by Pommerenke \cite{Po} and he proved that $V_g$ is bounded on the Hardy space $H^2$, if and only if $g \in BMOA$. Aleman and Siskakis \cite{AS1} extended this result for $H^p, 1 \leq p < \infty,$ and they considered later in \cite{AS2} the case of weighted Bergman spaces; see also \cite{PauPe}. We refer the reader to the memoir by Pel\'aez and R\"atty\"a \cite{PR} and the references therein. Volterra operators on weighted Banach spaces of holomorphic functions on the disc of type $H^\infty$ have been investigated recently in \cite{BCHMP} and this approach was influential in ours. Constantin started in \cite{C} the study of the Volterra operator on spaces of entire functions. She characterized the continuity of $V_g$ on the classical Fock spaces. Constantin and Pel\'aez \cite{CP} characterize the entire functions $g \in H(\C)$ such that $V_g$ is bounded or compact on a large class of Fock spaces induced by smooth radial weights.

Throughout the paper, a {\it weight} $v$   is a continuous function  $v: [0, \infty[ \to ]0,  \infty [$, which is non-increasing on $[0,\infty[$ and satisfies
$\lim_{r \rightarrow \infty} r^m v(r)=0$ for each $m \in \N$. If  necessary, 
we extend $v$ to $\C$ by $v(z):= v(|z|)$. For
such a weight, the {\it weighted Banach spaces of entire functions}
are defined by
\begin{center}
$H^\infty_v(\C) := \{ f \in H(\C) \ | \  \Vert f \Vert_v := \sup_{z \in \C} v(|z|) |f(z)| <
 \infty \}$,\\[3pt]
$H_v^0(\C) := \{ f \in H(\C) \ | \ \lim_{|z| \rightarrow \infty}
v(|z|)|f(z)|=0 \}$,
\end{center}
and they are endowed with the weighted sup norm $\Vert \cdot  \Vert_v .$ Clearly, $H_v^0(\C)$ is a closed subspace of $H_v^\infty(\C)$, which contains
the polynomials. Both are Banach spaces and the closed unit ball of $H^{\infty}_v(\C)$ is $\tau_{co}$-compact. The polynomials are contained and dense in $H_v^0(\C)$ but the monomials do not  in general 
form a Schauder basis, \cite{Lusky}. The Ces\`aro means of the Taylor polynomials satisfy $\Vert C_nf\Vert_v\leq \Vert f\Vert_v$  for each $f \in H^{\infty}_v(\C)$ and the sequence $(C_nf)_n$ is  
$\Vert \cdot \Vert_v$-convergent to $f$ when $f\in H^0_v(\C)$, see \cite{BiBoGal}. Clearly, changing the value of $v$ on a compact interval does not change the spaces and gives an equivalent norm. By \cite[Ex 2.2]{bisum}, the bidual of $H^0_v(\C)$ is isometrically isomorphic to $H^\infty_v(\C).$ Spaces of this type appear in the study of growth conditions of analytic functions and have been investigated
in various articles, see e.g. \cite{BiBoGal,BBT,Blasco_Galbis,galbis,Lusky,Lusk} and the references
therein.

The space $H_v^\infty(\C)$ is denoted as the weighted Fock space $\mathcal{F}^{\phi}_{\infty}$ of order infinity (i.e.\ with sup-norms) in \cite{CP} with $v(z)=\exp(-\phi(|z|))$, and $\phi:[0,\infty[ \rightarrow ]0,\infty[$ is a twice continuously differentiable  increasing function. The operator $V_g$ is denoted by $T_g$ in \cite{CP}.

For an entire function $f \in H(\C)$, we denote by $M(f,r):= \max\{|f(z)| \ | \ |z|=r\}$. Using the notation $O$ and $o$ of Landau, $f \in H_v^\infty(\C)$ if and only if $M(f,r)=O(1/v(r)), r \rightarrow \infty$, and $f \in H_v^0(\C)$ if and only if $M(f,r)=o(1/v(r)), r \rightarrow \infty$.

To clarify the notation, $f'$ denotes the usual complex derivative, if $f$ is an
analytic function, and the partial derivative with respect to the variable $r \in [0,\infty[$,
if $f$ is a weight (which is still defined in the entire plane)
or its inverse. By $C$, $C'$, $c$ (respectively, $C_n$) etc.
we denote  positive constants (resp. constant depending on the index $n$), the value of  which may vary from place to place.

\section{Multiplication operators} \label{multiplication}

In our study of Volterra operators we need the characterizations of boundedness and (weak) compactness of multiplication operators $M_h(f):=hf, h \in H(\C),$ between weighted Banach spaces $H^\infty_v(\C)$ of entire functions. These characterizations are well-known when the operators act on spaces of holomorphic  functions defined on the unit disc; see e.g. \cite{BDL} and \cite{ContH}.

We derive the results for the case of spaces of entire functions using
the so called \emph{associated weight} (see \cite{BBT}) as an important tool. For a weight $v$,  the associated weight $\tilde v$ is defined by

$$\tilde v (z):= \Big( \sup \big\{ |f(z)| \ | \  f \in H_v^{\infty}(\C), \| f \|_v \le 1 
\big\} \Big)^{-1}= \big( \| \delta_z \|_v \big)^{-1}, \; z \in \C,$$
where $\delta_z$ denotes the point evaluation of $z$. By \cite[Properties 1.2]{BBT} we know that the associated weight is continuous, radial, that $\tilde v \ge v >0$ holds  and that for each $z \in \D$ we can find  $f_z \in H_v^\infty$, $\Vert f_z\Vert_v = 1$ with $|f_z(z)| \tilde v(z)= 1$. It is also shown in \cite[Observation 1.12]{BBT} that $H_{\tilde v}^{\infty}(\C)$ coincides isometrically with  $H_v^{\infty}(\C)$.
Under the present assumptions on the weights, it is also true that $H_{\tilde v}^0(\C)$ coincides isometrically with  $H_v^0(\C)$. Indeed, since $v \leq \tilde{v}$ and $H_{\tilde v}^{\infty}(\C)$ coincides isometrically with  $H_v^{\infty}(\C)$, we find that $H_{\tilde v}^{0}(\C)$ is a closed subspace of $H_v^0(\C)$. By the assumption $\lim_{r \rightarrow \infty} r^m v(r)=0$ for each $m \in \N$,
and this implies $z^m \in H_{\tilde v}^{\infty}(\C)$ for all $m \in \N$. Therefore $z^m \in H_{\tilde v}^{0}(\C)$ for each $m \in \N$. Thus the polynomials $\mathcal{P}$ are contained in $H_{\tilde v}^{0}(\C)$. 
Since $\mathcal{P}$ is dense in $H_v^0(\C)$, the conclusion follows. Observe that we have also shown that  $\lim_{r \rightarrow \infty} r^m \tilde{v}(r)=0$ for each $m \in \N$.

A weight $v$ is called \textit{essential}, if there is $C>0$ such that $v(z) \leq \tilde{v}(z) \leq C v(z)$ for each $z \in \C$. A weight $v$ is essential, if and only if there is $c>0$ such that for each $z_0 \in \C$ there is $f_0 \in H(\C)$ with properties $|f_0(z_0)| \geq c/v(z_0)$ and $|f_0(z)| \leq 1/v(z)$ for all
$z \in \C$. It follows from \cite[Lemma 1]{Bor} or \cite[Theorem 17 and Lemma 46]{mmo} that the weight $v(z)=\exp(-\alpha |z|^p), z \in \C, \alpha >0, p>0, $ is essential; see Lemma \ref{lem3.5} and Remark
\ref{rem3.12} 
of this paper for more details.

\begin{proposition} \label{contmult}
Let $v$ and $w$ be weights. The following conditions are equivalent for an entire function $h \in H(\C)$: \begin{itemize}
\item[(1)] $M_h : H^{\infty}_v(\C) \rightarrow H^{\infty}_{w}(\C)$ is continuous.
\item[(2)] $M_h : H^{0}_v(\C) \rightarrow H^{0}_{w}(\C)$ is continuous.
\item[(3)] $\sup \frac{w(z)|h(z)|}{\tilde{v}(z)} < \infty$.
\item[(4)] $\sup \frac{\tilde{w}(z)|h(z)|}{\tilde{v}(z)} < \infty$.
\end{itemize}
\end{proposition}
\begin{proof}
$(1) \Rightarrow (2).$ By (1), $z^n h \in H^{\infty}_w(\C)$ for each $n \in \N$. This implies
$z^n h \in H^{0}_w(\C)$ for each $n \in \N$. Since the polynomials $\mathcal{P}$ are dense in
$H^{0}_v(\C)$, we have $M_h(H^{0}_v(\C)) \subset H^{0}_w(\C)$.

$(2) \Rightarrow (1).$ Fix $f \in H^{\infty}_v(\C)$. There is a sequence $(p_n)_n \in \mathcal{P}$ such that $\Vert p_n\Vert  \leq \Vert f\Vert $ for each $n \in \N$ and $p_n \rightarrow f$ in $H(\C)$ for the compact open topology (cf. \cite{BiBoGal}). Then $h p_n \rightarrow hf$ in
$H(\C)$ for the compact open topology, and by (2), there is $C>0$ such that $\Vert h p_n\Vert_w \leq C$ for each $n \in \N$. This implies $\Vert hf\Vert_w \leq C$; also $M_h(H^{\infty}_v(\C)) \subset H^{\infty}_{w}(\C)$ holds.

Clearly (4) implies (3).

$(3) \Rightarrow (1).$ Assume that $\frac{{w}(z)h(z)}{\tilde{v}(z)} \leq D$ for all $z \in \C$. Given $f \in H^{\infty}_v(\C)$ with $\Vert f\Vert_v \leq 1$, we have $|f| \leq 1/\tilde{v}$ on $\C$. Hence $w|hf| \leq \frac{w|h|}{\tilde{v}} \tilde{v}|f| \leq D$, and $hf \in H^{\infty}_{w}(\C)$. The conclusion follows from the closed graph theorem.

$(1) \Rightarrow (4).$ By (1), the transpose map $M_h^t:(H^{\infty}_{w}(\C))' \rightarrow (H^{\infty}_v(\C))'$ is continuous. It is easy to see that the set $\{\tilde{w} \delta_z \ | \ z \in \C \}$ is bounded in $(H^{\infty}_{w}(\C))'$. Since $M_h^t(\delta_z)=h(z) \delta_z$ for each $z \in \C$, one can find  $D>0$ such that
$$
\frac{\tilde{w}(z)|h(z)|}{\tilde{v}(z)} \leq \tilde{w}(z)|h(z)| \Vert \delta_z\Vert_v \leq \Vert M_h^t(\tilde{w}(z) \delta_z)\Vert  \leq D.
$$
\end{proof}

A sequence $(z_j)_j$ in $\C$ is called interpolating for $H^{\infty}_v(\C)$, if for every sequence $(\alpha_j)_j$ with $\sup_{j \in \N} v(z_j)|\alpha_j| < \infty$, there is $f \in H^{\infty}_v(\C)$ such that $f(z_j)=\alpha_j$ for each $j \in \N$. Examples of weights $v$ such that every discrete sequence in $\C$ has a subsequence, which is interpolating for $H^{\infty}_v(\C)$, are given in \cite[Proposition 9]{BiBo}. The result is based on \cite{mmo}. This property holds true for example for $v(z)= e^{-\alpha |z|^{p}}, \alpha >0, p>0$.

\begin{proposition} \label{compmult}
Let $v$ and $w$ be weights. The following conditions are equivalent for an entire function $h \in H(\C)$: \begin{itemize}
\item[(1)] $M_h : H^{\infty}_v(\C) \rightarrow H^{\infty}_{w}(\C)$ is compact.
\item[(2)] $M_h : H^{0}_v(\C) \rightarrow H^{0}_{w}(\C)$ is compact.
\item[(3)] $\lim_{|z| \rightarrow \infty} \frac{w(z)|h(z)|}{\tilde{v}(z)} = 0$.
\item[(4)] $\lim_{|z| \rightarrow \infty} \frac{\tilde{w}(z)|h(z)|}{\tilde{v}(z)} = 0$.
\end{itemize}
If, moreover every discrete sequence in $\C$ has a subsequence that is interpolating for $H^{\infty}_{\tilde v}(\C)$, these four conditions are also equivalent to
\begin{itemize}
\item[(5)] $M_h : H^{\infty}_v(\C) \rightarrow H^{\infty}_{w}(\C)$ is weakly compact.
\item[(6)] $M_h : H^{0}_v(\C) \rightarrow H^{0}_{w}(\C)$ is weakly compact.
\end{itemize}
\end{proposition}
\begin{proof}
If $M_h : H^{\infty}_v(\C) \rightarrow H^{\infty}_{w}(\C)$ is compact, then it is continuous and proposition \ref{contmult} implies that $M_h : H^{0}_v(\C) \rightarrow H^{0}_{w}(\C)$ is compact. Conversely, if $M_h : H^{0}_v(\C) \rightarrow H^{0}_{w}(\C)$ is compact, there is a compact subset $K$ of $H^{0}_{w}(\C)$ such that the unit ball $B_v^0$ of $H^{0}_v(\C)$ satisfies $M_h(B_v^0) \subset K$. By \cite{BiBoGal} or \cite{bisum}, since $K$ is compact for the compact open topology, the unit ball $B_v^\infty$ of $H^{\infty}_v(\C)$ satisfies $M_h(B_v^\infty) \subset M_h(\overline{B_v^0}) \subset \overline{M_h(B_v^0)} \subset K$, the closure taken for the compact open topology. This implies condition (1).

Clearly (4) implies (3). To show that (3) implies (1) it is enough to show that if a sequence $(f_k)_k$ is bounded in $H^{\infty}_v(\C)$ and $f_k \rightarrow 0$ for the compact open topology, then $h f_k \rightarrow 0$ in $H^{\infty}_{w}(\C)$ (see e.g. \cite[Section 2.4]{Sh}). To see that this holds, set $M:= \sup_{k \in \N} \Vert f_k\Vert_v$, and given $\varepsilon >0$, select $R>0$ such that, for $|z|>R$, $\frac{w(z)h(z)}{\tilde{v}(z)} < \varepsilon/(2M)$. Set $C:= \sup_{|z| \leq R} w(z)|h(z)|$. Since $f_k \rightarrow 0$ for the compact open topology, there is $k_0 \in \N$ such that, for $k \geq k_0$, we have $|f_k(z)| < \varepsilon/(2C)$ if $|z| \leq R$. Thus, if $k \geq k_0$ and $z \in \C$, we have $w(z)|h(z) f_k(z)| < \varepsilon$.

To complete the proof of the equivalence of the first four conditions, we show that (2) implies (4). By assumption (and Schauder's theorem), $M_h^t:(H^{0}_{w}(\C))' \rightarrow (H^{0}_v(\C))'$ is compact. Since $H^{0}_{w}(\C)=H^{0}_{\tilde w}(\C)$, for each $f \in H^{0}_{w}(\C)$ and each  $\varepsilon>0$ there is $R>0$ such that if $|z| >R$, then $|\tilde{w}(z) \delta_z(f)| = \tilde{w}(z) |f(z)| < \varepsilon$; i.e. $\lim_{|z| \rightarrow \infty} \tilde{w}(z) \delta_z(f) = 0$ for the weak* topology $\sigma((H^{0}_{w}(\C))',H^{0}_{w}(\C))$. Since the set $\{\tilde{w} \delta_z \ | \ z \in \C \}$ is bounded in $(H^{\infty}_{w}(\C))'$ and $M_h^t:(H^{0}_{w}(\C))' \rightarrow (H^{0}_v(\C))'$ is compact, the weak* convergence implies the norm convergence of $\lim_{|z| \rightarrow \infty} \tilde{w}(z) M_h^t(\delta_z(f)) = 0$ in $(H^{0}_v(\C))'$, which clearly implies condition (4).

Conditions (5) and (6) are both equivalent to $M_h(H^{\infty}_v(\C)) \subset H^{0}_{w}(\C)$  by \cite[p. 482]{DS} and the  Gantmacher theorem, since the bidual of $H^0_v(\C)$ is isometrically isomorphic to $H^\infty_v(\C)$ by \cite{bisum}. Conditions (1)--(4) clearly imply (5) and (6). Suppose now that condition (3) is not satisfied. We can find a discrete sequence $(z_j)_j$ in $\C$ and $\varepsilon >0$ such that $\frac{w(z_j|h(z_j)|}{\tilde{v}(z_j)} > \varepsilon$ for each $j \in \N$. By our assumption, there is a subsequence $(z_{j_k})_k$ of $(z_j)_j$ and there is $f \in H^{\infty}_v(\C)=H^{\infty}_{\tilde v}(\C)$ such that $f(z_{j_k})=1/\tilde{v}(z_{j_k})$ for each $k \in \N$. This implies $h f \notin H^{0}_{w}(\C)$ and (5) does not hold.

\end{proof}

\section{Volterra operators} \label{volterra}

In this section we present the main results, Theorems \ref{Volterracont}--\ref{th3.5a}. 
Concrete examples and applications  are given  Remark \ref{rem3.12} and its corollaries. 
The results
hold under some mild technical assumptions on the weights, and it will be convenient to 
first formulate some results for the inverse function $\varphi$ instead of the weight $w$ itself.
However, we start with results on  the continuity of the integration and differentiation operators. A thorough investigation of the continuity of these operators on weighted spaces of holomorphic functions have been undertaken by Harutyunyan and Lusky in \cite{HL}.

We consider the following setting. Let $\varphi: [0, \infty[ \to ]0,  \infty [$ be a continuous non-decreasing function,  which is $C^1$ on $[r_\varphi, \infty[$ for some $r_\varphi \geq 0$. We also suppose that the derivative $\varphi'$ is non-decreasing in $[r_\varphi, \infty[$, that $\varphi'(r_\varphi)>0$ and that $r^n = O(\varphi'(r))$ as $r \to \infty$ for each $n \in \N$. 
As a consequence of these assumptions,  $r^n = O(\varphi(r))$ as $r \to \infty$ for each $n \in \N$. Therefore $w_\varphi(z):=1/\varphi(|z|), z \in \C,$ is a weight. 
Clearly, also the function 
\begin{equation*}
u_\varphi(z):= 1/ \max\{ \varphi'(r_\varphi) , \varphi'(|z|) \} 
= 
\left\{
\begin{array}{ll}
1/  \varphi'(r_\varphi) \ \ , \ \ & |z| \leq r_\varphi \\
1/  \varphi'(|z|) \ \ , \ \ & |z| \geq r_\varphi 
\end{array}
\right.
\end{equation*}
is  a weight. We will keep this notation and these assumptions on the function $\varphi$ for  the rest of this section.
Here is one example: If $\varphi(r)=\exp(\alpha r^p), r \geq 0, \alpha >0, p > 0$, then $w_\varphi(z)= \exp(-\alpha |z|^p), z \in \C,$ and $u_\varphi(z)=  \alpha^{-1} p^{-1}|z|^{1- p}\exp(-\alpha |z|^p) $ for $|z|$ large enough.

Recall that the integration operator $J$ is the Volterra operator $V_g$ with $g$
as the identity mapping.

\begin{proposition} \label{intgroper}
The integration operators $J :H^{\infty}_{u_\varphi}(\C) \rightarrow H^{\infty}_{w_\varphi}(\C)$ and $J:H^{0}_{u_\varphi}(\C) \rightarrow H^{0}_{w_\varphi}(\C)$ are continuous.
\end{proposition}
\begin{proof}
By the proof of \cite[Lemma 2.1]{BeBoF}, it is enough to show that $J:H^{\infty}_{u_\varphi}(\C) \rightarrow H^{\infty}_{w_\varphi}(\C)$ is continuous. Fix $f \in H^{\infty}_{u_\varphi}(\C)$ with $\Vert f\Vert_{u_\varphi} \leq 1$. We have, for $z \in \C, |z| \geq r_\varphi,$
\begin{eqnarray}
& & |Jf(z)|  = \big| \int_0^1 f(tz) z dt \big| \leq \int_0^1  \frac{|z|}{ u_\varphi(t|z|)} dt = \int_0^{|z|}  \frac{1}{ u_\varphi(s)} ds 
\nonumber \\
& = & \int_0^{r_\varphi } \varphi'(r_\varphi) ds + \int_{r_\varphi}^{|z|} \varphi'(s) ds  = 
r_\varphi \varphi'(r_\varphi) + \varphi(|z|) - \varphi(r_\varphi).
\end{eqnarray}
This implies
$$w_\varphi(z) |Jf(z)| \leq 2 + \frac{r_\varphi \varphi'(r_\varphi)}{\varphi(r_\varphi)} \ \ \
\forall \ |z| \geq r_\varphi.$$
If $|z| \leq r_\varphi$, then $w_\varphi(z) |Jf(z)| \leq \frac{r_\varphi \varphi'(r_\varphi)}{\varphi(0)}$,
and the continuity of $J$ follows.
\end{proof}

\begin{proposition} \label{differoper}
If the function $\varphi$ is of smoothness $C^2$ on $[r_\varphi, \infty[$ for some $r_\varphi >0$ and it satisfies $\sup_{r \geq r_\varphi} \frac{\varphi''(r) \varphi(r)}{(\varphi'(r))^2} < \infty$
in addition to  the general assumptions of this section, then the differentiation operators $D: H^{\infty}_{w_\varphi}(\C) \rightarrow H^{\infty}_{u_\varphi}(\C)$ and
$D: H^{0}_{w_\varphi}(\C) \rightarrow H^{0}_{u_\varphi}(\C)$, $Df:= f'$, are continuous.
\end{proposition}
\begin{proof}
An entire function $f$ belongs to $H^{\infty}_{w_\varphi}(\C)$ if and only if $M(f,r)=O(\varphi(r))$, when 
$r \rightarrow \infty$. 
It follows  from \cite[Lemma 21]{CP} that this is equivalent to $M(f',r)=O(\varphi'(r))$ for $r \rightarrow \infty$, i.e. $Df \in H^{\infty}_{u_\varphi}(\C)$. The closed graph theorem then implies that $D: H^{\infty}_{w_\varphi}(\C) \rightarrow H^{\infty}_{u_\varphi}(\C)$ is continuous. Finally, the argument of \cite[Lemma 2.1]{BeBoF} implies that $D: H^{0}_{w_\varphi}(\C) \rightarrow H^{0}_{u_\varphi}(\C)$ is also continuous.
\end{proof}

The condition on the function $\varphi$ in Proposition \ref{differoper} corresponds to the condition $K_p$ in \cite{CP}; see  (3.4) in that paper. The argument behind \cite[Lemma 21]{CP} can be traced back at least to \cite[Theorem 2.1]{Pa}. A remarkable result of Hardy is used in \cite{Pa} to exhibit examples of functions $\varphi$ that satisfy the assumption of Proposition \ref{differoper}. For example one can take $\varphi(r):= r^a (\log)^b \exp(cr^d + k (\log r)^m)$, for large $r$, where $c>0, d>0$ or $c=0, k>0, m>1)$.

\vspace{.2cm}

As a consequence of Propositions \ref{intgroper} and \ref{differoper} we obtain the following result,
which gives a Littlewood-Paley-type formula for entire functions and growth estimates of infinite order
and which  should be compared with \cite[Theorem 10]{CP}.
The result follows directly from  the proven continuity of $J$ and $D$, and from $J Df=f-f(0)$.

\begin{corollary}\label{littlepaley}
Let $\varphi$ be as in Proposition \ref{differoper}.
An entire function $f$ satisfies $f \in H^{\infty}_{w_\varphi}(\C)$ (resp. $f \in H^{0}_{w_\varphi}(\C)$) if and only if $f' \in H^{\infty}_{u_\varphi}(\C)$ (resp. $f' \in H^{0}_{u_\varphi}(\C)$). Moreover, there are constants $C,C',C''>0$ such that, for each $f \in H^{\infty}_{w_\varphi}(\C)$,
$$\Vert f'\Vert_{u_\varphi} \leq C \Vert f\Vert_{w_\varphi}$$ and
$$\Vert f\Vert_{w_\varphi} \leq C'|f(0)| + C'' \Vert f'\Vert_{u_\varphi}.$$
\end{corollary}

We apply  the above mentioned results to the Volterra operator. 

\begin{theorem} \label{Volterracont}
Let $v$ be a weight and let $\varphi$ be as in Proposition \ref{differoper}.
The following conditions are equivalent for an entire function $g \in H(\C)$: 
\begin{itemize}
\item[(1)] $V_g : H^{\infty}_v(\C) \rightarrow H^{\infty}_{w_\varphi}(\C)$ is continuous.
\item[(2)] $V_g : H^{0}_v(\C) \rightarrow H^{0}_{w_\varphi}(\C)$ is continuous.
\item[(3)] $\sup_{|z| \geq r_\varphi} \frac{|g'(z)|}{\varphi'(|z|) \tilde{v}(z)} < \infty$.
\end{itemize}
\end{theorem}
\begin{proof}
Assume that condition (1) holds. By Proposition \ref{differoper}, the differentiation operator $D: H^{\infty}_{w_\varphi}(\C) \rightarrow H^{\infty}_{u_\varphi}(\C)$ is continuous. We can apply (1) and the identity $D V_g = M_{g'}$ to conclude that $M_{g'} : H^{\infty}_v(\C) \rightarrow H^{\infty}_{u_\varphi}(\C)$ is continuous. Now condition (3) follows from Proposition \ref{contmult}, since $u_{\varphi}(z)=1/\varphi'(|z|), |z| \geq r_\varphi$. Conversely, if condition (3) holds, the operator $M_{g'} : H^{\infty}_v(\C) \rightarrow H^{\infty}_{u_\varphi}(\C)$ is continuous by proposition \ref{contmult}. We apply Proposition \ref{intgroper} to get that $V_g=J \circ M_{g'}: H^{\infty}_v(\C) \rightarrow H^{\infty}_{w_\varphi}(\C)$ is continuous.

The equivalence of (2) and (3) is obtained in the same way, as a consequence of Propositions \ref{differoper}, \ref{intgroper} and \ref{contmult}.
\end{proof}

The corresponding statements on the compactness and weak compactness are as follows.
The proof is very similar to the one of Theorem \ref{Volterracont}, and it is a consequence of  Propositions \ref{differoper}, \ref{intgroper} and \ref{compmult}.

\begin{theorem} \label{Volterracomp}
Let $v$ be a weight and let  $\varphi$ be as in Proposition \ref{differoper}. 
The following conditions are equivalent for an entire function $g \in H(\C)$: 
\begin{itemize}
\item[(1)] $V_g : H^{\infty}_v(\C) \rightarrow H^{\infty}_{w_\varphi}(\C)$ is compact.
\item[(2)] $V_g : H^{0}_v(\C) \rightarrow H^{0}_{w_\varphi}(\C)$ is compact.
\item[(3)] $\lim_{|z| \rightarrow \infty} \frac{|g'(z)|}{\varphi'(|z|) \tilde{v}(z)} = 0$.
\end{itemize}
If, moreover every discrete sequence in $\C$ has a subsequence that is interpolating for $H^{\infty}_{\tilde v}(\C)$, these three conditions are also equivalent to
\begin{itemize}
\item[(4)] $V_g : H^{\infty}_v(\C) \rightarrow H^{\infty}_{w_\varphi}(\C)$ is weakly compact.
\item[(5)] $V_g : H^{0}_v(\C) \rightarrow H^{0}_{w_\varphi}(\C)$ is weakly compact.
\end{itemize}
\end{theorem}

Let us next reformulate the above results in terms of two quite arbitrary weights $v$ and $w$,
however, assuming some additional properties about the latter.

\begin{theorem} 
\label{th3.10}
Let $v$ and $w$ be weights, and assume that  for some constants $R > 0$, $C> 0$ and $0 < \delta \leq 1$, the following hold for $w$ on the 
interval  $]R,  \infty[$: $(i)$ $w$ is of smoothness $C^2$, $(ii)$
the function $|w'(r)| r^{1 + \delta}$ is non-increasing , $(iii)$  $ - \frac{ w(r) w''(r)}{w'(r)^2} \leq C $. The following conditions are equivalent for an entire function $g \in H(\C)$: 
\begin{itemize}
\item[(1)] $V_g : H^{\infty}_v(\C) \rightarrow H^{\infty}_{w}(\C)$ is continuous.
\item[(2)] $V_g : H^{0}_v(\C) \rightarrow H^{0}_{w}(\C)$ is continuous.
\item[(3)] $\sup_{|z| \geq R} \frac{w(z)^2 |g'(z)|}{w'(z) \tilde{v}(z)} 
< \infty$.
\end{itemize}
\end{theorem}

Notice that for a weight $w$ satisfying all assumptions of this theorem, the quantity on the left hand side of $(iii)$ may be negative, as it
is  for example in the case $w(r) = e^{-r}$. 

\begin{proof}
It is enough to show that the function $\varphi (z): = 1/ w(z)$ satisfies the assumptions 
of Theorem \ref{Volterracont}, since $w_\varphi =w$ and since the conditions (3) in Theorems
\ref{Volterracont} and \ref{th3.10} are the same. 

First, calculating the derivatives yields
$
\frac{\varphi'' \varphi}{(\varphi')^2} = 2 - \frac{w'' w}{(w')^2} ,
$
hence, the corresponding assumption of Theorem \ref{Volterracont} follows from $(iii)$. 
Second, in view of the beginning of this section, we should 
show that $\varphi'(r) \geq C_n r^n$ for all $n$, $r \geq R$. We 
have by  assumption $(ii)$
\begin{eqnarray}
& & w(r) = \int_r^\infty |w'(t)| dt = \int_r^\infty |w'(t)|t^{1 + \delta} t^{-1-\delta} dt
\nonumber \\
& \leq &  |w'(r)|r ^{1 + \delta} \int_r^\infty t^{-1-\delta} dt \leq C |w'(r)|r  
\ \ \ \forall r \geq R.
\label{3.89}
\end{eqnarray} 
Since $w $ is a weight, we have $w(r) \leq C_nr^{-n}$, so this and \eqref{3.89} imply
$$
\varphi'(r) = - \frac{w'(r)}{w(r)^2 } =  \frac1{w(r)} \frac{|w'(r)|}{w(r)} 
\geq C_n' r^n r^{-1 }  = C_n' r^{n -1}
$$
for arbitrary $n \geq 1$, $r \geq R$. This proves the claim. 
\end{proof}

The above argument and Theorem \ref{Volterracomp} lead also to the following
statements. 

\begin{theorem}
\label{th3.5a}
Let $v$  and $w$ be weights and assume that $w$ satisfies the conditions $(i)$--$(iii)$ of 
Theorem \ref{th3.10}. The following conditions are equivalent for an entire function $g \in H(\C)$: 
\begin{itemize}
\item[(1)] $V_g : H^{\infty}_v(\C) \rightarrow H^{\infty}_{w}(\C)$ is compact.
\item[(2)] $V_g : H^{0}_v(\C) \rightarrow H^{0}_{w}(\C)$ is compact.
\item[(3)] $\lim_{|z| \rightarrow \infty} \frac{w(z)^2 |g'(z)|}{w'(z) \tilde{v}(z)} = 0$.
\end{itemize}
If, moreover every discrete sequence in $\C$ has a subsequence that is interpolating for $H^{\infty}_{\tilde v}(\C)$, these three  conditions are also equivalent to
\begin{itemize}
\item[(4)] $V_g : H^{\infty}_v(\C) \rightarrow H^{\infty}_{w}(\C)$ is weakly compact.
\item[(5)] $V_g : H^{0}_v(\C) \rightarrow H^{0}_{w}(\C)$ is weakly compact.
\end{itemize}
\end{theorem}

We next formulate special cases of the above theorems.  For the proof
of the next result it suffices
to realize that the function $\varphi(r):= 1/w(r):= \exp(\alpha r^p) $ satisfies all the assumptions 
of Theorem \ref{Volterracont} and \ref{Volterracomp}, respectively.

\begin{corollary}\label{corexp1}
Let $v$ be a weight and let $w (r):= \exp(- \alpha r^p)$, where $ \alpha>0, p > 0$
are constants. The following conditions are equivalent for $g \in H(\C)$:
\begin{itemize}
\item[(1)] $V_g : H^{\infty}_v(\C) \rightarrow H^{\infty}_{w}(\C)$ is continuous.
\item[(2)] $V_g : H^{0}_v(\C) \rightarrow H^{0}_{w}(\C)$ is continuous.
\item[(3)] There exist a constant $C>0$ such that $|g'(z)| \leq C |z|^{p-1} \exp(\alpha |z|^p) \tilde{v}(|z|) $ for all  $z \in \C$, $|z| \geq 1$.
\end{itemize}

Moreover, the following conditions are also mutually equivalent for $g \in H(\C)$:
\begin{itemize}
\item[(1)] $V_g : H^{\infty}_v(\C) \rightarrow H^{\infty}_{w}(\C)$ is compact.
\item[(2)] $V_g : H^{0}_v(\C) \rightarrow H^{0}_{w}(\C)$ is compact.
\item[(3)] 
$|g'(z)| 
= o \big( |z|^{p-1} \exp(\alpha |z|^p) \tilde{v}(|z|) \big) $ as $|z| \to \infty$.
\end{itemize}
\end{corollary}

Let us finally consider weights, which are explicitly given in the form $v(r) = e^{-\phi(r)}$.
We make the assumptions that  the function $\phi: [0, \infty[ \to [0, \infty[$ is of
smoothness $C^2$, $\phi' > 0$ on $[0, \infty[$, and
\begin{equation}
r  \phi'(r) \to \infty  \  \ \mbox{as} \ r \to \infty \, ,  \ \mbox{and} \ \ \phi''(r) \leq (1 -\delta) \phi'(r)^2  \ \ \forall r  \in  [R_0, \infty[ , \label{3.43a}
\end{equation}
for some contants $R_0 \geq 0$ and  $0< \delta <1$.

\begin{remark}
The first condition means that $\phi$ grows faster than logarithmically. 
The second
condition is a monotonicity condition for $\phi''$. It is not satisfied for example by 
$\sin r$. Another such function with positive $\phi''$ can be constructed by
setting $\phi''(r) = e^{-r}$ on $[0, \infty[ \setminus K$, where $K $ is a compact set
containing some small neighbourhoods of the points $1,2,3, \ldots$ such that its measure
$|K|$  satisfies 
$|K| \leq 1/100$, and requiring
$\phi''(n) = 10$ for $n \in \N$, $\phi''(r) \leq 10 $ for all $r$. Then, defining $\phi'(r) =
\int_0^r \phi''(t)dt$ we have $\phi'(r) \leq 2$ for all $r$ and the second
inequality \eqref{3.43a} fails. 

On the other hand, there does not exist a positive function $\phi''$ 
with  $\phi'' (r) \geq c \phi'(r)^2$  ($c > 0$ constant) on an unbounded interval. 
This follows from the theory of the solutions of the nonlinear differential equation 
(think $g$ as $\phi'$)
\begin{equation}
g' = c g^2  .\label{3.43b}
\end{equation}
Any solution of \eqref{3.43b} with positive value $g(t) $ for some $t$, blows up 
at some finite value of $r$, i.e. it ceases to exist globally in $r$. For example, given $b >0$, the solution of \eqref{3.43b}
with $g(0) = b $ is $g(r) = b / (1 - cbr)$, which is defined only on the interval 
$[0, 1/(cb)[$. In view of this remark, the second condition in \eqref{3.43a} is indeed
quite mild. 
\end{remark}

These weights seem very useful for the theory presented in this paper, as 
we can see  from the following lemma.

\begin{lemma}
\label{lem3.5}
Let $\phi$ satisfy the conditions around \eqref{3.43a}. Then, 
the expression $v(r) = e^{-\phi(r)}$ is a weight and it satisfies $(ii)$ and $(iii)$ of  Theorem \ref{th3.10}. In addition, if there exists constants $r_0 , c  >0$ such that the inequality
\begin{equation}
\phi'(r) + r \phi''(r) \geq \frac{c}{r}
\label{3.43} 
\end{equation}
holds for $\phi$, for all $r \geq r_0$, then $v$ is also essential.
\end{lemma}

\begin{proof}
From the first relation in \eqref{3.43a} we get, for all $n \in \N$, for large enough $R > 0$,
$$
\phi(r) \geq \int\limits_R^r \phi'(t) dt \geq \int\limits_R^r \frac{n}{r } dt 
= n \log r  - n \log R, 
$$
hence, $v(r) \leq C r^{-n}$ and $v$ is a weight. Moreover, 
\begin{eqnarray}
& & \frac{d}{dr	} (- v'(r) r^{1+ \delta} ) = \Big( (1 + \delta)r^\delta \phi' (r) + r^{1+ \delta}
\big(\phi''(r) - \phi'(r)^2 \big) \Big)  e^{-\phi(r)}
\nonumber \\
& \leq &  \Big( (1 + \delta)r^\delta \phi' (r) - \delta r^{1+ \delta}
\phi'(r)^2  \Big)  e^{-\phi(r)}
 =  \phi' (r) r^\delta \Big( (1 + \delta) -  \delta r \phi'(r) \Big)   e^{-\phi(r)}  < 0 \nonumber
\end{eqnarray}
for large enough $r$, by both relations \eqref{3.43a}, since $r \phi'(r) \to \infty$ as $r \to
\infty$.  Hence, $v$ satisfies $(ii)$ of 
Theorem \ref{th3.10}. Finally,  again by \eqref{3.43a},
$$
- \frac{v(r)v''(r)}{v'(r)^2} = \frac{\phi''(r) - \phi'(r)^2}{\phi'(r)^2}
\leq - \delta, 
$$
so that $(iii)$ of Theorem \ref{th3.10} also holds.

To prove the essentialness of the weight we shall use a result of \cite{Bor}
(another approach would be  contained in \cite[Theorem 17 and Lemma 46]{mmo}),
and to this end we first observe that the function $\phi$ 
satisfies, for some constant $R > 0$, 
\begin{equation}
R \phi'(R) =1 .
\label{3.43c}
\end{equation}
Namely, $\lim_{r \to 0} r \phi'(r) = 0$, since $\phi'$ is assumed continuous on $[0, \infty[$,
and also $\lim_{r \to \infty} r \phi'(r) = \infty$ by the first relation \eqref{3.43a}.
Hence, \eqref{3.43c}  must hold for some $R$. 


The result \cite[Lemma 1]{Bor} gives an entire function $g$ 
with 
\begin{equation}
0 < c_1 \leq g(r) e^{-\sigma ( \log r) } \leq c_2 \ \ \forall r \geq 1,  \label{3.45}
\end{equation}
whenever $\sigma : [0, \infty[ \to \R$ is a $C^2$-function with $\sigma'(0) = 1$ and $\sigma''(r) \geq c_3 > 0$ 
for $r \geq 0$. Moreover, the Taylor coefficients of $g$ are positive, so the upper bound
in \eqref{3.45} can be extended so as 
$|g(z)| e^{-\sigma ( \log |z|) } \leq c_2$ for all $z \in \C$ with $|z| \geq 1$. 

In order to use this we define $\sigma(r) = \phi(R e^r)$, where $R > 0$ is
as in \eqref{3.43c}. As a consequence,  $\sigma'(0) = R \phi'(R) = 1$. Furthermore, 
$$
\sigma''(r) = R e^r \big( \phi'(R e^r) + R e^{r} \phi''(R e^r) \big)  ,
$$
which is larger than $c$ for all $r\geq 0$, by \eqref{3.43}. In view of \eqref{3.45},  the 
function $ g  $  has the properties
$$
g(r) \geq c_1 e^{ \sigma( \log r) } = c_1 e^{ \phi(R r) } \ , \ \ 
|g(z) |\leq g(|z|) \leq  c_2 e^{ \phi(R |z|) } 
$$
for $r \geq 1$, $|z| \geq 1$.
Thus, the entire function $h(z) := g (R^{-1} z) $  satisfies 
$
h(r) \geq c_1 e^{  \phi( r) }$, $|h(z)| \leq   c_2 e^{  \phi( |z|) } $
for $r \geq R$, $|z| \geq R$. This implies that the weight $e^{ - \phi( z) }$ is essential, since $\phi$ 
is a radially symmetric function.  
\end{proof}

\begin{remark} \label{rem3.12}
$(i)$ By Lemma \ref{lem3.5}, all functions $v(r) = \exp( {- \alpha r^p} )$, where $\alpha , p > 0$, are essential
weights:
it is easy to see that $\phi(r) = \alpha  r^{p}$ satisfies the conditions in \eqref{3.43a}, 
\eqref{3.43}. 

$(ii)$ The same is true for the more general functions 
$$
v(r) = \exp\big({- \alpha r^p + \beta (\log r)^q } \big)  \ \ , \ \ r \geq 2,
$$
with $\alpha , p ,q  > 0$, $\beta \in \R$, assuming the function is extended to $[0,2]$
properly.

$(iii)$ The case of the weights $v(r) = \exp( {- (\log r)^p} )$ with  $p > 1$ and $r $ as in 
$(ii)$,  is more subtle. 
We have $\phi'(r) =  p r^{-1} ( \log r)^{p-1}$,
$$
\phi''(r) = \frac{p}{r^2} \Big( (p-1) ( \log r)^{p-2} - ( \log r)^{p-1} \Big) .
$$
It follows that  \eqref{3.43a} holds for all $p > 1$, but \eqref{3.43} is valid
if and only if $ p\geq 2$. So, Lemma \ref{lem3.5} only permits us to conclude that  
these weights are essential for $p \geq 2$.
\end{remark}

As a consequence of Remark \ref{rem3.12} and Theorems \ref{th3.10}, \ref{th3.5a}
we obtain the following results. 
The next corollary should be compared with \cite[Corollary 25]{CP}.

\begin{corollary}
Let $v (r) = \exp( -\phi(r))$, where $\phi$ satisfies \eqref{3.43a} and \eqref{3.43}. 
Then, $V_g : H^{\infty}_v(\C) \rightarrow H^{\infty}_{v}(\C)$ is continuous if and only if
there exists a constant $C>0$ such that
$$
 |g'(z)| \leq C \phi'(|z|)  \ \ \forall z \in \C .
$$
In particular, if $v(r)=\exp(-\alpha r^p)$, $\alpha>0, p \geq 1$, then
$V_g : H^{\infty}_v(\C) \rightarrow H^{\infty}_{v}(\C)$ is continuous if and only if $g$ is a polynomial of degree less than or equal to the integer part of $p$. 
\end{corollary}

\begin{corollary}
If $v(r)=\exp(-\alpha r^p)$, $\alpha>0, p > 0$, then
$V_g : H^{\infty}_v(\C) \rightarrow H^{\infty}_v (\C)$ is compact if and only if it is weakly compact if and only if $g$ is a polynomial of degree less then or equal to the integer part of $p-1$.
\end{corollary}
\begin{proof}
Every discrete sequence in $\C$ has a subsequence that is interpolating for $H^{\infty}_v(\C)$ by \cite[Proposition 9]{BiBo}.
\end{proof}




\noindent \textbf{Authors' addresses:}%
\vspace{\baselineskip}%

Jos\'e Bonet: Instituto Universitario de Matem\'{a}tica Pura y Aplicada IUMPA,
Universitat Polit\`{e}cnica de Val\`{e}ncia,  E-46071 Valencia, Spain

email: jbonet@mat.upv.es \\

Jari Taskinen: Department of Mathematics and Statistics, P.O. Box 68, 
University of Helsinki, 00014 Helsinki, Finland.

email: jari.taskinen@helsinki.fi
\end{document}